\documentclass[12pt, reqno]{amsart}
\usepackage{amsmath, amsthm, amscd, amsfonts, latexsym, amssymb, graphicx, color, enumerate}
\usepackage[colorlinks,linkcolor=,citecolor = blue]{hyperref}

\makeatletter \oddsidemargin.9375in \evensidemargin \oddsidemargin
\newtheorem{theorem}{Theorem}[section]
\newtheorem{lemma}[theorem]{Lemma}
\newtheorem{proposition}[theorem]{Proposition}

\theoremstyle{definition}

\theoremstyle{remark}
\newtheorem{remark}[theorem]{Remark}
\numberwithin{equation}{section}

\newcommand{\p}[1]{\mathfrak{#1}}
\newcommand{\ppdots}[1]{{{\mathfrak{p}}_1}, \dots ,{{\mathfrak{p}}_{#1}}}
\newcommand{\qqdots}[1]{{{\mathfrak{q}}_1}, \dots ,{{\mathfrak{q}}_{#1}}}

\newcommand{\ass}[1]{\mathrm{Ass}(#1)}
\newcommand{\supp}[1]{\mathrm{Supp}(#1)}
\newcommand{\ann}[1]{\mathrm{Ann}(#1)}

\newcommand{\ppp}[1]{{{\mathfrak{p}}_1}^{r_1} \cdots {{\mathfrak{p}}_{#1}}^{r_{#1}}}
\newcommand{\ppn}[1]{{{\mathfrak{p}}_1}^{n_1} \cdots {{\mathfrak{p}}_{#1}}^{n_{#1}}}

\newcommand{\pp}[1]{{{\mathfrak{p}}_1} \cdots {{\mathfrak{p}}_{#1}}}

\newcommand{\pr}[1]{\mathfrak{#1}}

\newcommand{\assr}[1]{\mathrm{Ass}(#1)}
\newcommand{\prp}[1]{{{\mathfrak{p}}_1} \cdots {{\mathfrak{p}}_{#1}}}

\begin{document}

\setcounter{page}{1}

\title[Submodules Having the Same Prime Factorization]{Submodules Having the Same Generalized Prime Ideal Factorization}

\author[Thulasi, Duraivel, and Mangayarcarassy]{K. R. Thulasi$^{*}$, T. Duraivel, and S. Mangayarcarassy}

\thanks{{The first author was supported by INSPIRE Fellowship (IF170488) of the Department of Science and Technology (DST), Government of India.\\}{\scriptsize
\hskip -0.4 true cm MSC(2010): Primary: 13A05; Secondary: 13A15, 13E05, 13E15
$*$Corresponding author }}
\begin{abstract}
In our recent work, we introduced a generalization of the prime ideal factorization in Dedekind domains for submodules of finitely generated modules over Noetherian rings. In this article, we find conditions for the intersection of two submodules to have the same factorization as the submodules. We also find the relation between the factorizations of a submodule $N$ in an $R$-module $M$ and the ideal $\ann{M/N}$ in the ring $R$ and give a condition for their equality.
\end{abstract}

\maketitle
\section{Introduction}
Throughout this article, $R$ will be a commutative Noetherian ring with identity, and $M$ will be a finitely generated unitary $R$-module. The reference for standard terminology and notations will be \cite{C} and \cite{D}.

Let $N$ be a proper submodule of an $R$-module $M$. Then the ideal $(N : M)$ in $R$ is prime if for any $a \in R$ and $x \in M$, $ax \in N$ implies $a \in (N:M)$ or $x \in N$. We say $N$ is a $\pr p$-prime submodule of $M$ when $(N : M) = \pr p$, and in this case, $\assr{M/N} = \{\pr p\}$ \cite[Theorem~1]{CPLu}.

We say a submodule $K$ of $M$ is a $\pr p$-prime extension of $N$ in $M$ and denote it as $N \overset{\pr p} \subset K$ if $N$ is a $\pr p$-prime submodule of $K$. A $\pr p$-prime extension $K$ of $N$ is said to be maximal in $M$ if there is no $\pr p$-prime extension $L$ of $N$ in $M$ such that $L \supset K$. Since $M$ is Noetherian, maximal $\pr p$-prime extensions exist. It is proved that if $\pr p$ is a maximal element in $\assr{M/N}$, then $(N : \pr p)$ is the unique maximal $\pr p$-prime extension of $N$ in $M$ \cite[Theorem~11]{A} and it is called a regular $\pr p$-prime extension of $N$ in $M$.

A filtration of submodules $\mathcal{F} : N = M_0 \overset{{\pr p}_1}\subset M_1 \subset \cdots \overset{{\pr p}_n}\subset M_n = M$ is called a regular prime extension (RPE) filtration of $M$ over $N$ if each $M_i$ is a regular ${\pr p}_i$-prime extension of $M_{i-1}$ in $M$, $1 \leq i \leq n$. RPE filtrations are defined and studied in \cite{A}. Also, RPE filtrations are weak prime decompositions defined in \cite{Dress}.

The following result shows that $\ass{M/N}$ is precisely the set of prime ideals occurring in any RPE filtration of $M$ over $N$.
\begin{lemma}\cite[Proposition~14]{A}\label{lemma*}
Let $N$ be a proper submodule of $M$. If $ N = M_0 \overset{{\p p}_1}\subset M_1 \subset \cdots \subset M_{n-1} \overset{{\p p}_n}\subset M_n = M $ is an RPE filtration of $M$ over $N$, then $\ass{M/M_{i-1}} = \{{\p p}_{i},\ldots,{\p p}_n\}$ for $1 \leq i \leq n$. In particular, $\ass{M/N} = \{{\p p}_1,\ldots,{\p p}_n\}$.
\end{lemma}

The following lemma characterizes the submodules occurring in an RPE filtration.
\begin{lemma}\cite[Lemma~3.1]{B}\label{lemma1}
Let $N$ be a proper submodule of $M$. If $N = M_0 \overset{{\p p}_1}\subset M_1 \subset \cdots \subset M_{i-1}\overset{{\p p}_i}\subset M_i \overset{{\p p}_{i+1}}\subset M_{i+1}\subset \cdots \overset{{\p p}_n}\subset M_n = M $ is an RPE filtration of $M$ over $N$, then $M_i = \{ x \in M \mid \pp i x \subseteq N \}$ for $ 1 \leq i \leq n$. In other words, $M_i = (N : \pp i)$ for $ 1 \leq i \leq n$.
\end{lemma}
The occurrences of two prime ideals in an RPE filtration can be interchanged provided they satisfy the following condition.

\begin{lemma}\cite[Lemma~20]{A}\label{interchange}
Let $N$ be a proper submodule of $M$ and $N = M_0 \subset \cdots \subset M_{i-1}\overset{{\p p}_i}\subset M_i \overset{{\p p}_{i+1}}\subset M_{i+1}\subset \cdots \subset M_n = M $ be an RPE filtration of $M$ over $N$. If ${\p p}_{i+1} \not\subseteq {\p p}_{i}$, then there exists a submodule $K_i$ of $M$ such that $N = M_0 \subset \cdots \subset M_{i-1}\overset{{\p p}_{i+1}}\subset K_i \overset{{\p p}_{i}}\subset M_{i+1}\subset \cdots \subset M_n =M$ is an RPE filtration of $M$ over $N$.
\end{lemma}

\begin{remark}\label{rem_interchange}
So for every reordering ${\p p}_1', \dots , {\p p}_n' $ of $\ppdots n$ with ${\p p}_i' \not \subset {\p p}_j'$ for $i<j$, we can get an RPE filtration $$N \overset{{\p p}_1'}\subset M_1' \overset{{\p p}_2'}\subset M_2' \subset \cdots \subset M_{n-1}' \overset{{\p p}_n'}\subset M_n' = M.$$ In particular, if ${\p p}_i$ is minimal among $\{ \ppdots n \}$, then we can have an RPE filtration
\begin{multline*}
N=M_0 \overset{{\p p}_1}\subset M_1 \subset \cdots \overset{{\p p}_{i-1}} \subset M_{i-1}\overset{{\p p}_{i+1}}\subset K_i \overset{{\p p}_{i+2}} \subset K_{i+1} \subset \cdots \overset{{\p p}_n} \subset K_{n-1} \overset{{\p p}_i}\subset M
\end{multline*}
\cite[Remark~2.5]{B}. In general, if ${\p p}_i$ occurs $r$ times in an RPE filtration, then there exists an RPE filtration
\begin{multline*}
N = M_0 \overset{{\p p}_1}\subset M_1 \subset \cdots \overset{{\p p}_{i-1}} \subset M_{i-1} \overset{{\p p}_i}\subset M_i \overset{{\p p}_i} \subset M_{i+1} \subset \cdots \\ \overset{{\p p}_i} \subset M_{i+r-1} \overset{{\p p}_{i+1}}\subset M_{i+r} \overset{{\p p}_{i+2}} \subset \cdots \subset M_n = M
\end{multline*}
where ${\p p}_j \not\subseteq {\p p}_i$ for $j<i$.
\end{remark}

It is proved that in any RPE filtration of $M$ over $N$, the number of occurrences of each prime ideal is unique \cite[Theorem~22]{A}. Hence, if $N = M_0 \overset{{\pr p}_1}\subset M_1 \subset \cdots \overset{{\pr p}_n}\subset M_n = M $ is an RPE filtration, then the product $\prp n$ is uniquely defined for $N$ in $M$ and it is called the generalized prime ideal factorization of $N$ in $M$. We write ${\mathcal{P}}_M(N)= \prp n$ and in short, we call it the prime factorization of $N$ in $M$. Generalized prime ideal factorization of submodules is defined and studied in \cite{E}.

In \cite{E} it was observed that in a finitely generated module over a Noetherian ring, distinct submodules may have the same prime factorization. For example, in $k[x,y]$, the ideals $(x^2,y)$ and $(x,y^2)$ have the same prime factorization $(x,y)^2$ \cite[Example~2.5]{E}. We see that their intersection $(x^2,xy,y^2)$ also has the prime factorization $(x,y)^2$. In this article we show that this need not always be true. For submodules $N$ and $K$ of an $R$-module $M$ with ${\mathcal{P}}_{M}(N) = {\mathcal{P}}_{M}(K)$, we find conditions for ${\mathcal{P}}_{M}(N \cap K) = {\mathcal{P}}_{M}(N)$. We also compare the prime factorizations of a submodule $N$ in $M$ and the ideal $\ann{M/N}$ in $R$. We show that the product ${\mathcal{P}}_{M}(N)$ is a multiple of the product ${\mathcal{P}}_{R}(\ann{M/N})$ and give a sufficient condition for ${\mathcal{P}}_{M}(N) = {\mathcal{P}}_{R}(\ann{M/N})$.

We use the following lemmas.

\begin{lemma}\cite[Lemma~2.8]{B}\label{lemma}
If $N \overset{\p p} \subset K$ is a regular $\p p$-prime extension in $M$, then for any submodule $L$ of $M$, $N \cap L \overset{\p p} \subset K \cap L$ is a regular $\p p$-prime extension in $L$ when $N \cap L \neq K \cap L$.
\end{lemma}
Hence intersecting a regular prime extension with a submodule gives a regular prime extension whenever equality does not occur.

\begin{lemma}\label{intersectionresult}
Let $N_1 \overset{\p p} \subset N_2$ and $K_1 \overset{\p p} \subset K_2$ be regular prime extensions in $M$. If $N_1 \cap K_1 \neq N_2 \cap K_2$, then $N_1 \cap K_1 \overset{\p p}\subset N_2 \cap K_2$ is a regular prime extension in $M$.
\end{lemma}
\begin{proof}
We have $\p p \subseteq (N_1 \cap K_1 : N_2 \cap K_2)$ since $\p p N_2 \subseteq N_1$ and $\p p K_2 \subseteq K_1$. Now let $a \in (N_1 \cap K_1 : N_2 \cap K_2)$ and $x \in N_2 \cap K_2\setminus N_1 \cap K_1$. Then $ax \in N_1 \cap K_1$. Without loss of generality, we assume $x \notin N_1$. Then since $N_1 \overset{\p p} \subset N_2$ is a prime extension, $ax \in N_1$ implies $a \in (N_1 : N_2) = \p p$. Hence $(N_1 \cap K_1 : N_2 \cap K_2) = \p p$.

Let $ax \in N_1 \cap K_1$ for some $x \in N_2 \cap K_2$, $a \in R$ such that $x \notin N_1 \cap K_1$. Without loss of generality, assume $x \notin N_1$. Then $ax \in N_1 \cap K_1 \subseteq N_1$ implies $a \in (N_1 : N_2) = \p p$. Therefore, $N_2 \cap K_2$ is a $\p p$-prime extension of $N_1 \cap K_1$, and $\{ \p p \} = \ass{\frac{N_2 \cap K_2}{N_1 \cap K_1}} \subseteq \ass{\frac{M}{N_1 \cap K_1}}$.

Let $L$ be any $\p p$-prime extension of $N_1 \cap K_1$ in $M$ and let $x \in L$. Then $\p p x \subseteq N_1 \cap K_1$, i.e., $x \in (N_1 : \p p) \cap (K_1 : \p p) = N_2 \cap K_2$ [Lemma \ref{lemma1}], and hence $L \subseteq N_2 \cap K_2$. So $N_2 \cap K_2$ is a maximal $\p p$-prime extension of $N_1 \cap K_1$ in $M$.

Suppose $\p q \in \ass{M/{N_1 \cap K_1}}$ and $\p q \supseteq \p p$. Then $\p q = (N_1 \cap K_1 : x)$ for some $x \in M$. Since $\p p \subseteq \p q$, $\p p x \subseteq N_1 \cap K_1$. That is, $x \in (N_1 : \p p) \cap (K_1 : \p p) = N_2 \cap K_2$. This implies, $\p q \in \ass{\frac{N_2 \cap K_2}{N_1 \cap K_1}} = \{ \p p \}$. That is, $\p q = \p p$. Therefore, $\p p$ is a maximal element in $\ass{M/{N_1 \cap K_1}}$, and hence $N_2 \cap K_2$ is a regular $\p p$-prime extension of $N_1 \cap K_1$ in $M$.
\end{proof}
The next lemma gives a condition satisfied by the prime factorization of a submodule.
\begin{lemma}\label{moduleresult}
Let $N$ be a submodule of $M$ having ${\mathcal{P}}_M(N) = \pp n$. Then $\pp {i-1}{\p p}_{i+1} \cdots {\p p}_n M \not \subseteq N$ whenever ${\p p}_i \not \supset {\p p}_j$ for every $1 \leq j \leq n$.
\end{lemma}
\begin{proof}
There exists an RPE filtration $$N = N_0 \overset{{\p p}_1}\subset N_1 \subset \cdots \subset N_{i-1}\overset{{\p p}_i}\subset N_i \subset \cdots \subset N_{n-1} \overset{{\p p}_n}\subset N_n = M.$$ If for some $i$, ${\p p}_i \not \supset {\p p}_j$ for $j = 1, \dots, n$, by Remark \ref{rem_interchange} we can have an RPE filtration
$$N=N_0 \overset{{\p p}_1}\subset N_1 \subset \cdots \overset{{\p p}_{i-1}} \subset N_{i-1}\overset{{\p p}_{i+1}}\subset K_i \overset{{\p p}_{i+2}} \subset K_{i+1} \subset \cdots \overset{{\p p}_n} \subset K_{n-1} \overset{{\p p}_i}\subset M.$$
Then by Lemma \ref{lemma1}, $K_{n-1} = \{ x \in M \mid \pp {i-1}{\p p}_{i+1} \cdots {\p p}_n x \subseteq N\}$. So $\pp {i-1}{\p p}_{i+1} \cdots {\p p}_n M \subseteq N$ would imply $M\subseteq K_{n-1}$, which is a contradiction. Hence $\pp {i-1}{\p p}_{i+1} \cdots {\p p}_n M \not \subseteq N$.
\end{proof}

Lemma \ref{moduleresult} does not hold if ${\p p}_j \subset {\p p}_i$ for some $j$. Let $R = k[x,y,z] / (xy - z^2)$. Then ${\p p}_1 = (\overline{x},\overline{y},\overline{z})$ and ${\p p}_2 = (\overline{x},\overline{z})$ are prime ideals in $R$ and ${{\p p}_2}^2$ has the RPE filtration $${{\p p}_2}^2 = (\overline{x}^2, \overline{x}\overline{y}, \overline{x}\overline{z}) \overset{{\p p}_1} \subset (\overline{x}) \overset{{\p p}_2} \subset (\overline{x},\overline{z})\overset{{\p p}_2} \subset R.$$ If $M=R$ and $N={{\p p}_2}^2$, then we have ${\mathcal{P}}_M(N) = {\p p}_1 {{\p p}_2}^2$. But ${{\p p}_2}^2 M = N$.

\begin{remark}\label{eg1_prime}
For a prime ideal $\p p$ in $R$, we have ${\mathcal{P}}_{R}(\p p) = \p p$ \cite[Example~2.2]{E}. In fact, the only ideal in $R$ having $\p p$ as its generalized prime ideal factorization is $\p p$ itself. For suppose ${\mathcal{P}}_{R}(\p a) = \p p$ for an ideal $\p a$ in $R$. Then $\p a \overset{\p p} \subset R$ is an RPE filtration, which implies $(\p a : R) = \p p$. So $\p p \subseteq \p a \subseteq \sqrt{\p a}$. Also, $\ass{R/\p a} = \{ \p p \}$ and therefore, $\p a$ is $\p p$-primary, which gives $\sqrt{\p a} = \p p$. Hence we get $\p a = \p p$. 
\end{remark}

\section{Main Results}

\begin{proposition}\label{intersectionassoc}
Let $\p a, \p b$ be ideals in $R$. If ${\mathcal{P}}_{R}(\p a) = {\mathcal{P}}_{R}(\p b)$ and it is a product of at most two prime ideals, then ${\mathcal{P}}_{R}(\p a \cap \p b) = {\mathcal{P}}_{R}(\p a)$. 
\end{proposition}
\begin{proof}
If ${\mathcal{P}}_{R}(\p a) = {\mathcal{P}}_{R}(\p b) = \p p$ for some prime ideal $\p p$ in $R$, then by Remark \ref{eg1_prime}, $\p a = \p b = \p p$. Therefore, $\p a \cap \p b = \p p \overset{\p p} \subset R$ is an RPE filtration, and hence ${\mathcal{P}}_{R}(\p a \cap \p b) = \p p = {\mathcal{P}}_{R}(\p a)$.

If ${\mathcal{P}}_{R}(\p a) = {\mathcal{P}}_{R}(\p b) = {\p p}_1{\p p}_2$ for some prime ideals ${\p p}_1, {\p p}_2$ in $R$, we have RPE filtrations
$$\p a \overset{{\p p}_1}\subset {\p a}_1 \overset{{\p p}_2} \subset R$$
$$\p b \overset{{\p p}_1}\subset {\p b}_1 \overset{{\p p}_2} \subset R.$$
By Remark \ref{eg1_prime}, ${\p a}_1 = {\p b}_1 = {\p p}_2$. Since $\p a \subset {\p a}_1$ and $\p b \subset {\p b}_1$, $\p a \cap \p b \subset {\p p}_2 = {\p a}_1 \cap {\p b}_1$. So, using Lemma \ref{intersectionresult}, we have the RPE filtration $\p a \cap \p b \overset{{\p p}_1}\subset {\p p}_2 \overset{{\p p}_2} \subset R$. Therefore ${\mathcal{P}}_{R}(\p a \cap \p b) = {\mathcal{P}}_{R}(\p a)$.
\end{proof}

If ${\mathcal{P}}_{R}(\p a) = {\mathcal{P}}_{R}(\p b)$ and it is a product of more than two prime ideals, then ${\mathcal{P}}_{R}(\p a \cap \p b)$ need not be equal to ${\mathcal{P}}_{R}(\p a)$. For example, in $R = k[x,y,z]$, let $\p a = (x^2,y^2,xy,xz)$ and $\p b = (x^2,y^2,xy,yz)$. Then we have ${\mathcal{P}}_{R}(\p a) = {\mathcal{P}}_{R}(\p b) = (x,y,z)(x,y)(x,y)$ since there are RPE filtrations
\begin{center}
$(x^2,y^2,xy,xz) \overset{(x,y,z)} \subset (x,y^2) \overset{(x,y)} \subset (x,y) \overset{(x,y)} \subset R$
\end{center}
and \begin{center} $(x^2,y^2,xy,yz) \overset{(x,y,z)} \subset (x^2,y) \overset{(x,y)} \subset (x,y) \overset{(x,y)} \subset R$.\end{center}

But ${\mathcal{P}}_{R}(\p a \cap \p b) = (x,y)(x,y)$ since $$\p a \cap \p b = (x^2,y^2,xy) \overset{(x,y)} \subset (x,y) \overset{(x,y)} \subset R$$ is the RPE filtration of $R$ over $\p a \cap \p b$.

For submodules $N$ and $K$ of $M$ having the same prime factorization $\pp n$, ${\mathcal{P}}_{M}(N \cap K)$ need not be equal to ${\mathcal{P}}_{M}(N)$ even for $n=2$. For example, in the $\mathbb{Z}$-module $\mathbb{Z} \oplus \mathbb{Z}$, we have the RPE filtrations $2\mathbb{Z} \oplus 0 \overset{2\mathbb{Z}} \subset \mathbb{Z} \oplus 0 \overset{0} \subset \mathbb{Z} \oplus \mathbb{Z}$ and
$0 \oplus 2\mathbb{Z} \overset{2\mathbb{Z}}\subset 0 \oplus \mathbb{Z} \overset{0}\subset \mathbb{Z} \oplus \mathbb{Z}$. So the submodules $2\mathbb{Z} \oplus 0$ and $0 \oplus 2\mathbb{Z}$ have the same prime factorization. But $(2\mathbb{Z} \oplus 0) \cap (0 \oplus 2\mathbb{Z}) = 0 \oplus 0$ and ${\mathcal{P}}_{\mathbb{Z} \oplus \mathbb{Z}}(0 \oplus 0) = 0 \neq {\mathcal{P}}_{\mathbb{Z} \oplus \mathbb{Z}}(2\mathbb{Z} \oplus 0)$.

Now we find conditions for ${\mathcal{P}}_{M}(N \cap K) = {\mathcal{P}}_{M}(N)$ for submodules $N$ and $K$ of $M$ with ${\mathcal{P}}_{M}(N) = {\mathcal{P}}_{M}(K)$.

\begin{proposition}\label{moduleassoc}
Let $N$ and $K$ be submodules of $M$ with ${\mathcal{P}}_{M}(N) = {\mathcal{P}}_{M}(K) = \pp n$. Then ${\mathcal{P}}_{M}(N \cap K) = {\mathcal{P}}_{M}(N)$ if ${\p p}_i \not \supset {\p p}_j$ for every $i \neq j$.
\end{proposition}
\begin{proof}
We prove by induction on $n$. If $n=1$, then we have ${\mathcal{P}}_{M}(N) = {\mathcal{P}}_{M}(K) = \p p$ for some prime ideal $\p p$ in $R$. So we have RPE filtrations $N \overset{\p p}\subset M$ and $K \overset{\p p}\subset M$. By Lemma \ref{intersectionresult}, $M$ is a regular ${\p p}$-prime extension of $N \cap K$, and hence $N \cap K \overset{\p p}\subset M$ is an RPE filtration. Therefore ${\mathcal{P}}_{M}(N \cap K) = \p p = {\mathcal{P}}_{M}(N)$.

Now let $n>1$, and assume the result is true for $n-1$. Suppose ${\mathcal{P}}_{M}(N) = {\mathcal{P}}_{M}(K) = \pp n$, where ${\p p}_i \not \supset {\p p}_j$ for every $1 \leq i,j \leq n$, $i \neq j$. Then we have RPE filtrations
$$N = {N}_0 \overset{{\p p}_1}\subset {N}_1 \subset \cdots \overset{{\p p}_{n-1}}\subset {N}_{n-1} \overset{{\p p}_n}\subset {N}_n = M$$
$$K = {K}_0 \overset{{\p p}_1}\subset {K}_1 \subset \cdots \overset{{\p p}_{n-1}}\subset {K}_{n-1} \overset{{\p p}_n}\subset {K}_n = M.$$
Since ${\mathcal{P}}_{M}({N}_1) = {\mathcal{P}}_{M}({K}_1) = {\p p}_2 \cdots {\p p}_n$, by induction hypothesis we get ${\mathcal{P}}_{M}({N}_1 \cap {K}_1) = {\mathcal{P}}_{M}({N}_1) = {\p p}_2 \cdots {\p p}_n$. So we have an RPE filtration $${N}_1 \cap {K}_1 \overset{{\p p}_2} \subset L_2 \subset \cdots \subset {L}_{n-1} \overset{{\p p}_n}\subset L_n= M.$$

If ${N}_1 \cap {K}_1 = N \cap K$, then ${\mathcal{P}}_{M}(N \cap K) = {\p p}_2 \cdots {\p p}_n$, which implies ${\p p}_2 \cdots {\p p}_n M \subseteq N \cap K$. But since ${\p p}_1 \not \supset {\p p}_j$ for every $1 \leq j \leq n$, by Lemma \ref{moduleresult}, we have ${\p p}_2 \cdots {\p p}_n M \not \subseteq N$ and ${\p p}_2 \cdots {\p p}_n M \not \subseteq K$, which is a contradiction. So $N \cap K \subset {N}_1 \cap {K}_1$. Then by Lemma \ref{intersectionresult}, $N \cap K \overset{{\p p}_1}\subset {N}_1 \cap {K}_1$ is a regular ${\p p}_1$-prime extension in $M$. Therefore, $$N \cap K \overset{{\p p}_1}\subset {N}_1 \cap {K}_1 \overset{{\p p}_2} \subset L_2 \subset \cdots \subset {L}_{n-1} \overset{{\p p}_n}\subset L_n= M$$ is an RPE filtration of $M$ over $N \cap K$. Hence ${\mathcal{P}}_{M}(N \cap K) = \pp n = {\mathcal{P}}_{M}(N)$.
\end{proof}

For a proper submodule $N$ of $M$, let $\p a =\ann{M/N}$. Next, we compare the prime factorizations $\mathcal{P}_R(\p a)$ and $\mathcal{P}_M(N)$.
\begin{lemma}\label{lemma(N:K)}
Let $N$ be a proper submodule of $M$. Then for any submodule $K$ of $M$, $\ass{\frac{R}{(N:K)}} \subseteq \ass{M/N}$.
\end{lemma}
\begin{proof}
Let $\p p \in \ass{\frac{R}{(N:K)}}$. Then $\p p = ((N:K) : a) = \ann{\frac{aK + N}{N}}$ for some $a \in R$. So $\p p$ is a minimal element in $\supp{\frac{aK + N}{N}}$, and therefore $\p p \in \ass{\frac{aK + N}{N}} \subseteq \ass{M/N}$.
\end{proof}

\begin{lemma}\label{lemma(N:K)P^n}
Let $N$ be a submodule of $M$ with $\mathcal{P}_M(N) = {\p p}^n$. Then for any submodule $K$ of $M$, $\mathcal{P}_R((N:K)) = {\p p}^r$, where $r \leq n$.
\end{lemma}
\begin{proof}
By Lemma \ref{lemma(N:K)}, $\ass{\frac{R}{(N:K)}} \subseteq \ass{M/N} = \{ \p p \}$. So $\mathcal{P}_R((N:K)) = {\p p}^r$ for some $r$. Suppose $r > n$. Then ${\p p}^{r-1} \subseteq{\p p}^n$. Also, by Lemma \ref{moduleresult}, ${\p p}^{r-1} \not\subseteq (N:K)$. So there exists $a \in {\p p}^{r-1}$ such that $aK \not\subseteq N$. Since $a \in {\p p}^n$ and ${\p p}^n M \subseteq N$, we get $aK \subseteq N$, a contradiction. Therefore $r \leq n$.
\end{proof}

\begin{lemma}\label{ass(R/a)}
Let $N$ be a proper submodule of $M$ and $\p a =\ann{M/N}$. Then $\ass{R/{\p a}} \subseteq \ass{M/N}$. Also, $\ass{R/{\p a}} = \ass{M/N}$ if every prime ideal in $\ass{M/N}$ is isolated.
\end{lemma}
\begin{proof}
Taking $K=M$ in Lemma \ref{lemma(N:K)}, we get $\ass{R/{\p a}} \subseteq \ass{M/N}$. Suppose every prime ideal in $\ass{M/N}$ is isolated. Since $\ann{R/\p a} = \ann{M/N}$, $\supp{R/{\p a}} = \supp{M/N}$, and they have the same set of minimal elements. Therefore, we have
\begin{multline*}
\ass{M/N} = \min{\ass{M/N}} = \min{\supp{M/N}} = \\ \min{\supp{R/\p a}} =  \min{\ass{R/\p a}} \subseteq \ass{R/\p a}.
\end{multline*}
Hence $\ass{R/{\p a}} = \ass{M/N}$.
\end{proof}

In the above lemma, the condition that all the prime ideals in $\ass{M/N}$ must be isolated cannot be omitted. For, if $M$ is the $\mathbb{Z}$-module $\mathbb{Z} \oplus \mathbb{Z}$ and $N=2\mathbb{Z} \oplus 0$, then $\ass{M/N} = \{ 2\mathbb{Z} , 0 \}$. But since $\p a = \ann{M/N} = 0$, $\ass{R/{\p a}} = \{ 0 \}$.

So $\mathcal{P}_M(N)$ need not be equal to $\mathcal{P}_R(\p a)$.

\begin{theorem}\label{thm_divides}
Let $N$ be a proper submodule of $M$ and $\p a = \ann{M/N}$. Then $\mathcal{P}_M(N)$ is a multiple of $\mathcal{P}_R(\p a)$ as a product of prime ideals.
\end{theorem}
\begin{proof}
If $\mathcal{P}_M(N) = {\p p}^n$ for some prime ideal $\p p$ in $R$, then taking $K=M$ in Lemma \ref{lemma(N:K)P^n} we get $\mathcal{P}_R(\p a) = {\p p}^r$, where $r \leq n$. Hence $\mathcal{P}_M(N)$ is a multiple of $\mathcal{P}_R(\p a)$.

Now let $\mathcal{P}_M(N) = \ppn k$, where ${\p p}_i$'s are distinct primes. Then we have an RPE filtration
\begin{multline}\label{mod_m/n_k}
N \overset{{\p p}_1}\subset N_1 \subset \cdots \overset{{\p p}_1}\subset N_{n_1} \overset{{\p p}_2}\subset N_{n_1+1} \subset \cdots \overset{{\p p}_i}\subset N_{n_1 + \cdots + n_i}\\ \overset{{\p p}_{i+1}}\subset \cdots \overset{{\p p}_k}\subset N_{n_1 + \cdots + n_k} = M
\end{multline}
such that ${\p p}_i \not\subseteq {\p p}_j$ for $1 \leq i < j \leq k$. Let $\p a  \overset{{\p q}_1}\subset {\p a}_1 \subset \cdots \overset{{\p q}_m} \subset {\p a}_m = R$ be an RPE filtration of $R$ over $\p a$. Then $\{ \qqdots m \} = \ass{R/ \p a} \subseteq \ass{M/N}$ [Lemma \ref{ass(R/a)}]. So $\mathcal{P}_R(\p a) = \ppp k$, where $r_i \geq 0$. Then by Remark \ref{rem_interchange}, we can have an RPE filtration
\begin{equation}\label{ide_r/a_k}
{\p a} \overset{{\p p}_1}\subset {\p a}_1 \subset \cdots \overset{{\p p}_1}\subset {\p a}_{r_1} \overset{{\p p}_2}\subset {\p a}_{r_1+1} \subset \cdots \overset{{\p p}_i}\subset {\p a}_{r_1 + \cdots + r_i} \overset{{\p p}_{i+1}}\subset \cdots  \overset{{\p p}_k}\subset {\p a}_{r_1 + \cdots + r_k} = R.
\end{equation}
Suppose $r_i > n_i$ for some $i$ and let $i$ be the least such integer. Let $N_i' = N_{n_1 + \cdots + n_i}$ and ${\p a}_i' = {\p a}_{r_1 + \cdots + r_{i-1} + n_i}$. Then $N_i' = (N : \ppn i)$ and ${\p a}_i' = (\p a :  \ppp {i-1}{{\p p}_i}^{n_i})$ by Lemma \ref{lemma1}. Let $a \in {\p a}_i'$. Then $$a \ppn i \subseteq a \ppp {i-1}{{\p p}_i}^{n_i} \subseteq \p a = (N : M).$$
That is, $aM \subseteq (N : \ppn i) = N_i'$. Therefore, $a \in (N_i' : M)$, and this implies ${\p a}_i' \subseteq (N_i' : M)$.

We have ${\p a}_i' = {\p a}_{r_1 + \cdots + r_{i-1} + n_i} \subset {\p a}_{r_1 + \cdots + r_i}$ from the filtration (\ref{ide_r/a_k}) since $n_i < r_i$. So we have ${\p p}_i \in \ass{R/{\p a}_i'}$, and therefore for some $b \in R$, ${\p p}_i = ({\p a}_i' : b)$. Clearly $({\p a}_i' : b) \subseteq ((N_i' : M) : b)$. If $b \in (N_i' : M)$, this implies $\ppn i bM \subseteq N$. That is, $\ppn i b \subseteq \p a$. Then
$$({{\p p}_1}^{n_1 - r_1} {{\p p}_2}^{n_2 - r_2}  \cdots {{\p p}_{i-1}}^{(n_{i-1}) - (r_{i-1})}) \, b \subseteq (\p a : \ppp {i-1}{{\p p}_i}^{n_i} ) = {\p a}_i'.$$
By assumption, $n_j \geq r_j$ for $j=1, \dots, i-1$. If $n_j = r_j$ for $j=1, \dots, i-1$, then this implies $b \in {\p a}_i'$, i.e., $({\p a}_i' : b)=R$, a contradiction. If $n_j > r_j$ for some $j \in \{ 1, \dots, i-1 \}$, then 
$${{\p p}_1}^{n_1 - r_1} {{\p p}_2}^{n_2 - r_2}  \cdots {{\p p}_{i-1}}^{(n_{i-1}) - (r_{i-1})} \; \subseteq ({\p a}_i' : b) = {\p p}_i$$
implies ${\p p}_j \subseteq {\p p}_i$ for some $j<i$, a contradiction. So $b \notin (N_i' : M)$. Then ${\p p}_i \subseteq ((N_i' : M) : b) \subseteq \p q$ for some $\p q \in \ass{\frac{R}{(N_i' : M)}}$. From Lemma \ref{lemma(N:K)} and (\ref{mod_m/n_k}) we get $\ass{\frac{R}{(N_i' : M)}} \subseteq \ass{M/N_i'} = \{{\p p}_{i+1}, \dots , {\p p}_k \}$. This implies ${\p p}_i \subseteq {\p p}_l$ for some $l \in \{ i+1, \dots, k\}$, which is not true. Therefore $r_i \leq n_i$ for all $i$. Hence $\mathcal{P}_M(N)$ is a multiple of $\mathcal{P}_R(\p a)$.
\end{proof}

\begin{theorem}\label{thm_divides=}
Let $N$ be a proper submodule of $M$ and $\p a = \ann{M/N}$. If every prime ideal in $\ass{M/N}$ is isolated, then $\mathcal{P}_R(\p a) = \mathcal{P}_M(N)$.
\end{theorem}
\begin{proof}
Let $\mathcal{P}_M(N) = \ppn k$, where ${\p p}_i$'s are distinct primes. Then by Lemma \ref{ass(R/a)} and Theorem \ref{thm_divides} we have $\mathcal{P}_R(\p a) = \ppp k$, where $1 \leq r_i \leq n_i$ for $1 \leq i \leq k$. So by Remark \ref{rem_interchange} we have RPE filtrations
\begin{multline*}
N \overset{{\p p}_1}\subset N_1 \subset \cdots \overset{{\p p}_1}\subset N_{n_1} \overset{{\p p}_2}\subset N_{n_1+1} \subset \cdots \overset{{\p p}_i}\subset N_{n_1 + \cdots + n_i} \\ \overset{{\p p}_{i+1}}\subset \cdots \overset{{\p p}_k}\subset N_{n_1 + \cdots + n_k} = M;
\end{multline*}
$${\p a} \overset{{\p p}_1}\subset {\p a}_1 \subset \cdots \overset{{\p p}_1}\subset {\p a}_{r_1} \overset{{\p p}_2}\subset {\p a}_{r_1+1} \subset \cdots \overset{{\p p}_i}\subset {\p a}_{r_1 + \cdots + r_i} \overset{{\p p}_{i+1}}\subset \cdots  \overset{{\p p}_k}\subset {\p a}_{r_1 + \cdots + r_k} = R.$$
Note that $\ass{R/{\p a}_{r_1}} = \{{\p p}_2, \dots , {\p p}_k \}$ [Lemma \ref{lemma*}].

Suppose $r_i < n_i$ for some $i$. Since ${\p p}_i \not \subseteq {\p p}_j$ whenever $i \neq j$, without loss of generality, we assume that $i=1$ by applying Remark \ref{rem_interchange}.

By Lemma \ref{lemma1}, $N_{r_1} = (N : {{\p p}_1}^{r_1} )$ and ${\p a}_{r_1} = (\p a : {{\p p}_1}^{r_1} )$. So for $a \in R$, we have
\begin{tabular}{l l l}
$a \in {\p a}_{r_1}$ & $\Leftrightarrow$ & ${{\p p}_1}^{r_1} a \subseteq \p a = (N:M)$ \\
 & $\Leftrightarrow$ & ${{\p p}_1}^{r_1} a M \subseteq N$ \\
 & $\Leftrightarrow$ & $aM \subseteq (N : {{\p p}_1}^{r_1} ) = N_{r_1}$\\
 & $\Leftrightarrow$ & $a \in (N_{r_1} : M)$.
\end{tabular}

Therefore ${\p a}_{r_1} = (N_{r_1} : M)$. Since every prime ideal in $\ass{M/N_{r_1}}$ is isolated, by Lemma \ref{ass(R/a)}, $$\ass{M/N_{r_1}} = \ass{R/{(N_{r_1} : M)}} = \ass{R/{\p a}_{r_1}} = \{{\p p}_2, \dots , {\p p}_k \}.$$ Since
$$N_{r_1}\overset{{\p p}_1}\subset N_{r_1 + 1} \subset \cdots \overset{{\p p}_1}\subset N_{n_1} \overset{{\p p}_{2}}\subset \cdots \overset{{\p p}_k}\subset N_{n_1 + \cdots + n_k} = M$$ is an RPE filtration, ${\p p}_1 \in \ass{M/N_{r_1}}$, a contradiction. Therefore $r_i = n_i$ for all $i$. Hence $\mathcal{P}_R(\p a) = \mathcal{P}_M(N)$.
\end{proof}


%
\vskip 0.4 true cm

\bibliographystyle{amsplain}

\bigskip
\bigskip

\noindent {\footnotesize {\bf  K. R. Thulasi }\; \\
Department of Mathematics, Pondicherry University, Pondicherry, India.\\
 {\tt thulasi.3008@gmail.com}

\noindent {\footnotesize {\bf  T. Duraivel }\; \\
Department of Mathematics, Pondicherry University, Pondicherry, India.\\
 {\tt tduraivel@gmail.com}
 
\noindent {\footnotesize {\bf  S. Mangayarcarassy }\; \\
Department of Mathematics, Puducherry Technological University, Pondicherry, India.\\
 {\tt dmangay@pec.edu}

\end{document}